\documentclass[a4paper,11pt]{article}
\usepackage{mathrsfs, amsmath, amsxtra, amssymb, latexsym, amscd, amsthm}

\newtheorem{thm}{Theorem}[section]
\newtheorem{cor}{Corollary}[section]
\newtheorem{lem}{Lemma}[section]
\newtheorem{prop}{Proposition}[section]
\theoremstyle{definition}
\newtheorem{defn}{Definition}[section]
\newtheorem{exam}{Example}[section]

\theoremstyle{remark}
\newtheorem{rem}{Remark}[section]
\numberwithin{equation}{section}

\def\N{{\rm I\kern-0.16em N}}
\def\R{{\rm I\kern-0.16em R}}
\def \E{{\rm I\kern-0.16em E}}
\def\P{{\rm I\kern-0.16em P}}
\def\F{{\rm I\kern-0.16em F}}
\def\B{{\rm I\kern-0.16em B}}
\def\C{{\rm I\kern-0.46em C}}
\def\G{{\rm I\kern-0.50em G}}

\begin{document}

\title
{Rates of convergence in the CLT for nonlinear statistics under relaxed moment conditions}
\author{Nguyen Tien Dung\thanks{Department of Mathematics, VNU University of Science, Vietnam National University, Hanoi, 334 Nguyen Trai, Thanh Xuan, Hanoi, Vietnam. Email: dung@hus.edu.vn}
}
\date{\today}
\maketitle
\begin{abstract}This paper is concerned with normal approximation under relaxed moment conditions using Stein's method. We obtain the explicit rates of convergence in the central limit theorem for (i) nonlinear statistics with finite absolute moment of order $2+\delta\in(2,3];$ (ii) nonlinear statistics with vanishing third moment and finite absolute moment of order $3+\delta\in(3,4].$ When applied to specific examples, these rates are of the optimal order $O(n^{-\frac{\delta}{2}})$ and $O(n^{-\frac{1+\delta}{2}}).$ Our proof are based on the covariance identify formula and simple observations about the solution of Stein's equation.
\end{abstract}
\section{Introduction}

Let $X=(X_1,X_2,...,X_n)$ be a vector of independent random variables (not necessarily identically distributed). We consider the problem of normal approximations for nonlinear statistics of the form
\begin{equation}\label{lesser3}
F=F(X_1,...,X_n).
\end{equation}
We recall that this is one of the most fundamental problems in the theory of mathematical statistics. The main task is to investigate the rate of convergence in the central limit theorem (CLT) for $F.$ When $F$ has finite absolute moments of order $p\geq3,$ this problem has been well studied. A significant amount of normal approximation results for $F$ and its special forms can be found in the literature. The reader can consult the monograph \cite{Chen2007} for a detailed representation of this topic. 

We now consider the case, where $F$ only has finite absolute moments of order $2+\delta\in(2,3].$ This case is more difficult to study and it requires some new ideas. It seems to us that not too many general results can be found in the literature.
\begin{itemize}
\item For the partial sum of $\mathbb{R}$-valued independent random variables $S_n:=\sum\limits_{k=1}^n(X_k-\mu_k)$, the classical result proved by Lyapunov in 1901 says that $S_n/\sigma_n$ converges in distribution to a standard normal random variable $N$ if
\begin{equation}\label{iif2}
\lim\limits_{n\to \infty}\frac{1}{\sigma_n^{2+\delta}}\sum\limits_{k=1}^nE|X_k-\mu_k|^{2+\delta}=0,
 \end{equation}
where $\mu_k:=E[X_k]$ and $\sigma^2_n:=\sum\limits_{k=1}^nE|X_k-\mu_k|^{2}.$ Sixty five years latter, the rate of convergence in Lyapunov's central limit theorem was established by Bikjalis \cite{Bikjalis1966} and Ibragimov \cite{Ibragimov1966}. They obtained the following error bound
\begin{equation}\label{iif}
d_1(S_n/\sigma_n,N)\leq \frac{C_\delta}{\sigma_n^{2+\delta}}\sum\limits_{k=1}^nE|X_k-\mu_k|^{2+\delta}
\end{equation}
for some constant $C_\delta$ depending only on $\delta,$ where $d_1$ denotes the Wasserstein distance. If, in addition, the random variables are identically distributed, then $d_1(S_n,N)=O(n^{-\frac{\delta}{2}})$ and this rate is optimal, see e.g. \cite{Rio2009} for a short survey.
\item More recently, the optimal rate of convergence for certain nonlinear statistics was also obtained in \cite{Bentkus2009,Chen2004,Chen2007}.  Meanwhile, Bentkus et al. \cite{Bentkus2009} focus on $U$-statistics, Chen and Shao \cite{Chen2004,Chen2007} investigate the sum of locally dependent random variables and the nonlinear statistics that can be written as $F=W+\Delta,$ where $W$ is a linear statistic and $\Delta$ is an error term.
 \item Surprisingly, to the best of our knowledge, a systematic study for nonlinear statistics of the general form (\ref{lesser3}) is still missing. This is the first motivation of the present paper. In fact, our Theorem \ref{lesser4} will partially fill up this gap by providing explicit bounds on Wasserstein distance for the rate of convergence.
\end{itemize}
Another motivation of this paper comes from the vanishing third moment phenomenon discussed in Section 4.8 of \cite{Chen2011}. This phenomenon says that, under additional moment assumptions, the standard convergence rate $O(n^{-\frac{1}{2}})$ can be improved to $O(n^{-1}).$ Let us recall the following.
\begin{prop}\label{jjggk4} Let $N$ be a standard normal distribution and $X,X_1,...,X_n$ be independent and identically distributed mean zero, variance one random variables with $EX^3=0$ and $E|X|^4$ finite. Then, for $\bar{S}_n:=n^{-1/2}(X_1+...+X_n),$  we have
$$d_{\mathcal{H}_4}(\bar{S}_n,N):=\sup\limits_{h\in \mathcal{H}_4}|E[h(\bar{S}_n)]-E[h(N)]|\leq \frac{1}{24n}\left(11+E|X|^4\right),$$
where $\mathcal{H}_4:=\{h:\mathbb{R}\to \mathbb{R}:\max\limits_{0\leq k\leq 4}\|h^{(k)}\|_\infty\leq 1\}$ and $\|.\|_\infty$ denotes the supremum norm.
\end{prop}
The moment condition $E|X|^4<\infty$ is the best possible one to achieve the rate $O(n^{-1}).$ So there are two open questions arising here: (i) Find the rate of convergence under moment condition $E|X|^{3+\delta}<\infty$ for some $\delta\in(0,1],$ (ii) Generalize the above phenomenon for nonlinear statistics (\ref{lesser3}). Our Theorem \ref{ltser4a} will provide a complete answer to both those questions. In fact, when applied to $\bar{S}_n,$ we obtain the rate $O(n^{-\frac{1+\delta}{2}}).$

Powerful as it is, Stein's method will be the main tool to prove our Theorems \ref{lesser4} and \ref{ltser4a}. Recall that this method was proposed by Stein in 1970's and since then, many different techniques have been developed to use it. The present paper will continue employing the technique based on difference operators which was used in our recent paper \cite{dungnt2019}. The key allowing us to relax moment conditions is simple observations about the solution of Stein's equation, see Propositions \ref{tt1} and \ref{tt2}.

The rest of the paper is organized as follows. Our main results (Theorems \ref{lesser4} and \ref{ltser4a}) are described in Section \ref{9jk2}. Some illustrative examples with detailed computations are given in Section \ref{hj48y}. Proofs of the main theorems are given in Section \ref{8i2}. Some useful moment inequalities are provided in Section \ref{hj48y9}.
\section{The main results}\label{9jk2}Throughout this paper let $N$ denote a standard normal random variable. To measure the distance to normality of a random variable $G,$ we will use the following two distances

\noindent $\bullet$ $d_1$-distance (or Wasserstein distance) defined by
$$d_1(G,N):=\sup\limits_{h\in \mathcal{C}^1:\|h'\|_\infty\leq 1}|E[h(G)]-E[h(N)]|.$$
\noindent $\bullet$ $d_2$-distance defined by
$$d_2(G,N):=\sup\limits_{h\in \mathcal{C}^2:\|h'\|_\infty,\|h''\|_\infty\leq 1}|E[h(G)]-E[h(N)]|,$$
where $\mathcal{C}^k(k\geq 1)$ is the space of $k$-times differentiable real-valued functions on $\mathbb{R}$ and $\|.\|_\infty$ denotes the supremum norm.

We now describe the main results of this paper. Let $\mathcal{X}$ be a measurable space and $X=(X_1,X_2,...,X_n)$ be a vector of independent random variables, defined on some probability space $(\Omega,\mathfrak{F},P)$ and taking values in $\mathcal{X}.$ Let $U:\mathcal{X}^n\to \mathbb{R}$ be a measurable function, the random variable $U=U(X)$ is called a nonlinear statistic. We introduce the $\sigma$-fields
\begin{align*}
&\mathcal{F}_0:=\{\emptyset,\Omega\}\,\,\,\text{and}\,\,\,\mathcal{F}_i:=\sigma(X_k,k\leq i),\,\,i=1,...,n
\end{align*}
and
\begin{align*}
&\mathcal{G}_{n+1}:=\{\emptyset,\Omega\}\,\,\,\text{and}\,\,\,\mathcal{G}_i:=\sigma(X_k,k\geq i),\,\,i=1,...,n.
\end{align*}
\begin{defn}\label{kod2ol} Given a random variable $U=U(X)\in L^1(P),$ we define the difference operators $\mathfrak{D}_i$ by
$$\mathfrak{D}_iU=U-E_i[U],\,\,1\leq i\leq n,$$
where $E_i$ denotes the expectations with respect to $X_i.$ In addition, for each $\alpha\in [0,1],$ we define
$$\mathfrak{D}^{(\alpha)}_i U:=\alpha E[\mathfrak{D}_iU|\mathcal{F}_i]+(1-\alpha)E[\mathfrak{D}_iU|\mathcal{G}_i],\,\,1\leq i\leq n.$$
\end{defn}
We note that the difference operators $\mathfrak{D}_i$ are very useful in the study of concentration inequalities., see e.g. \cite{Bobkov2017}. In the context of normal approximations, recent papers \cite{Decreusefond2019,dungnt2019} have successfully used  those operators for nonlinear statistics with finite fourth moment. The following two properties of $\mathfrak{D}_i$ will be used in our present work.
\begin{prop}\label{77h4} Let $U=U(X)$ be in $L^p(P)$ for some $p\geq 1.$ Then, for every $i=1,...,n,$ we have

\noindent (i) $E|\mathfrak{D}_iU|^p\leq 2^p E|U|^p,$ 

\noindent (ii)  $E|\mathfrak{D}^{(\alpha)}_i U|^p\leq E|\mathfrak{D}_iU|^p$ for all $\alpha\in [0,1].$
\end{prop}
\begin{proof}The point $(i)$ was already proved in Proposition 2.2 of \cite{dungnt2019}. Let us prove the point $(ii).$ Using the discrete H\"older inequality we get
\begin{align*}
|\mathfrak{D}^{(\alpha)}_i U|&\leq \alpha^{\frac{1}{p}} |E[\mathfrak{D}_iU|\mathcal{F}_i]|\alpha^{\frac{p-1}{p}}+(1-\alpha)^{\frac{1}{p}}|E[\mathfrak{D}_iU|\mathcal{G}_i]|(1-\alpha)^{\frac{p-1}{p}}\\
&\leq \left(\alpha |E[\mathfrak{D}_iU|\mathcal{F}_i]|^p+(1-\alpha)|E[\mathfrak{D}_iU|\mathcal{G}_i]|^p\right)^{\frac{1}{p}}.
\end{align*}
Then, by Lyapunov's inequality,
\begin{align*}
E|\mathfrak{D}^{(\alpha)}_i U|^p&\leq \alpha E|E[\mathfrak{D}_iU|\mathcal{F}_i]|^p+(1-\alpha)E|E[\mathfrak{D}_iU|\mathcal{G}_i]|^p\\
&\leq \alpha E|\mathfrak{D}_iU|^p+(1-\alpha)E|\mathfrak{D}_iU|^p\\
&=E|\mathfrak{D}_iU|^p.
\end{align*}
This finishes the proof of Proposition.
\end{proof}
The next theorem is the first main result of the present paper where we provide explicit bounds on the Wasserstein distance $d_1$ for nonlinear statistics with  finite absolute moment of order lesser than $3.$
\begin{thm}\label{lesser4}Fix $\delta\in (0,1]$ and let $F=F(X)\in L^{2+\delta}(P)$ be centered with $\sigma^2:=Var(F)\in(0,\infty).$ For any $\alpha,\beta\in[0,1],$ we always have
\begin{align}
&d_1(F/{\sigma},N)\leq \frac{2}{\sigma^{2+\delta}} \sum\limits_{k=1}^nE\left[(|\mathfrak{D}_kF|^{\delta}+E_k|\mathfrak{D}_kF|^{\delta})|\mathfrak{D}^{(\beta)}_kZ^{(\alpha)}|\right]\notag\\
&\hspace{3cm}+\frac{2^{1+\delta}}{\sigma^{2+\delta}} \sum\limits_{k=1}^n E\left[(|\mathfrak{D}_kF|^{1+\delta}+E_k|\mathfrak{D}_kF|^{1+\delta})|\mathfrak{D}_k^{(\alpha)}F|\right]\label{8je4a}\\
&\leq \frac{4}{\sigma^{2+\delta}} \big(\sum\limits_{k=1}^nE|\mathfrak{D}_kF|^{2+\delta}\big)^{\frac{\delta}{2+\delta}}
\big(\sum\limits_{k=1}^nE|\mathfrak{D}_kZ^{(\alpha)}|^{\frac{2+\delta}{2}}\big)^{\frac{2}{2+\delta}}+\frac{2^{2+\delta}}{\sigma^{2+\delta}} \sum\limits_{k=1}^n E|\mathfrak{D}_kF|^{2+\delta},\label{8je4ab}
\end{align}
where $Z^{(\alpha)}:=\sum\limits_{k=1}^n\mathfrak{D}_kF\mathfrak{D}_k^{(\alpha)}F.$
\end{thm}
Since the topology induced by Wasserstein distance $d_1$ is stronger than that of convergence in distribution, we obtain the following CLT for nonlinear statistics which can be considered as a natural generalization of the classical Lyapunov central limit theorem. It should be noted that, in the case of partial sums $S_n,$ the condition (\ref{o1}) is itself satisfied and the condition (\ref{o2}) is exact (\ref{iif2}).
\begin{cor}[Lyapunov's CLT]\label{8h2k} Let $(F_n)_{n\geq 1}$ be a sequence of nonlinear statistics in $L^{2+\delta}(\Omega)$ with $\sigma_n^2:=Var(F_n)\in (0,\infty).$ We put
$$L_n:=\frac{1}{\sigma_n^{2+\delta}}\sum\limits_{k=1}^nE|\mathfrak{D}_kF_n|^{2+\delta}.$$
Assume that there exist $\alpha\in[0,1]$ and $\varepsilon\in(0,\delta]$ such that
\begin{equation}\label{o1}
\sup\limits_{n\geq 1}\frac{L_n^{\frac{\delta-\varepsilon}{2}}}{\sigma_n^{2+\delta}}\sum\limits_{k=1}^nE|\mathfrak{D}_kZ^{(\alpha)}_n|^{\frac{2+\delta}{2}}<\infty,
\end{equation}
where $Z_n^{(\alpha)}:=\sum\limits_{k=1}^n\mathfrak{D}_kF_n\mathfrak{D}_k^{(\alpha)}F_n.$ Then, $\frac{F_n-E[F_n]}{\sigma_n}$ converges in distribution to a standard normal random variable as $n\to\infty$ if
\begin{equation}\label{o2}
\lim\limits_{n\to \infty}L_n=0.
\end{equation}
\end{cor}
\begin{proof}Follows directly from the bound (\ref{8je4ab}) and the following relation
$$
\frac{4}{\sigma_n^{2+\delta}} \big(\sum\limits_{k=1}^nE|\mathfrak{D}_kF_n|^{2+\delta}\big)^{\frac{\delta}{2+\delta}}
\big(\sum\limits_{k=1}^nE|\mathfrak{D}_kZ_n^{(\alpha)}|^{\frac{2+\delta}{2}}\big)^{\frac{2}{2+\delta}}
=4L_n^{\frac{\varepsilon}{2+\delta}}
\left(\frac{L_n^{\frac{\delta-\varepsilon}{2}}}{\sigma_n^{2+\delta}}\sum\limits_{k=1}^nE|\mathfrak{D}_kZ^{(\alpha)}_n|^{\frac{2+\delta}{2}}\right)^{\frac{2}{2+\delta}}
.$$
\end{proof}
The second main result of the present paper is formulated in the next theorem where we investigate the vanishing third moment phenomenon under relaxed moment condition $E|F|^{3+\delta}<\infty.$
\begin{thm}\label{ltser4a}Fix $\delta\in (0,1]$ and let $F=F(X)\in L^{3+\delta}(P)$ be centered with $\sigma^2:=Var(F)\in(0,\infty)$ and $E[F^3]=0.$ For any $\alpha,\beta,\gamma\in[0,1],$ we always have
\begin{align}
d_2(F/\sigma,N)&\leq  \frac{4}{\sigma^{3+\delta}} \sum\limits_{k=1}^nE\left[(|\mathfrak{D}_kF|^{\delta}+E_k|\mathfrak{D}_kF|^{\delta})|\mathfrak{D}^{(\gamma)}_kZ^{(\alpha,\beta)}|\right]\notag\\
&+\frac{2^{2+\delta}}{\sigma^{3+\delta}}\sum\limits_{k=1}^nE\left[(|\mathfrak{D}_kF|^{1+\delta}+E_k|\mathfrak{D}_kF|^{1+\delta})|\mathfrak{D}^{(\beta)}_kZ^{(\alpha)}|\right]
\notag\\
&+\frac{2^{2+\delta}}{\sigma^{3+\delta}}\sum\limits_{k=1}^nE\left[(|\mathfrak{D}_kF|^{2+\delta}+E_k|\mathfrak{D}_kF|^{2+\delta})|\mathfrak{D}^{(\alpha)}_kF|\right]
\label{kofd4}\\
&\leq \frac{8}{\sigma^{3+\delta}}\big(\sum\limits_{k=1}^nE|\mathfrak{D}_kF|^{3+\delta}\big)^{\frac{\delta}{3+\delta}}
\big(\sum\limits_{k=1}^nE|\mathfrak{D}_kZ^{(\alpha,\beta)}|^{\frac{3+\delta}{3}}\big)^{\frac{3}{3+\delta}}\notag\\
&+\frac{2^{3+\delta}}{\sigma^{3+\delta}}\big(\sum\limits_{k=1}^nE|\mathfrak{D}_kF|^{3+\delta}\big)^{\frac{1+\delta}{3+\delta}}
\big(\sum\limits_{k=1}^nE|\mathfrak{D}_kZ^{(\alpha)}|^{\frac{3+\delta}{2}}\big)^{\frac{2}{3+\delta}}+\frac{2^{3+\delta}}{\sigma^{3+\delta}}\sum\limits_{k=1}^nE|\mathfrak{D}_kF|^{3+\delta},
\label{kofd44}
\end{align}
where $Z^{(\alpha)}$ is as in Theorem \ref{lesser4} and
$$Z^{(\alpha,\beta)}:=\sum\limits_{k=1}^n\mathfrak{D}_kF\mathfrak{D}^{(\beta)}_kZ^{(\alpha)}-\frac{1}{2}\sum\limits_{k=1}^n(|\mathfrak{D}_kF|^2
+E_k|\mathfrak{D}_kF|^2)\mathfrak{D}^{(\alpha)}_kF.$$
\end{thm}

\begin{rem} By the fundamental inequality $(|a_1|+...+|a_N|)^m\leq N^{m-1}(|a_1|^m+...+|a_N|^m)$ we have
\begin{align*}
E|Z^{(\alpha)}|^\frac{2+\delta}{2}&\leq   n^\frac{\delta}{2}\sum\limits_{k=1}^nE|\mathfrak{D}_kF\mathfrak{D}_k^{(\alpha)}F|^\frac{2+\delta}{2}\\
&\leq   n^\frac{\delta}{2}\sum\limits_{k=1}^n\sqrt{E|\mathfrak{D}_kF|^{2+\delta}E|\mathfrak{D}_k^{(\alpha)}F|^{2+\delta}}\,\,\,\text{by the H\"older inequality}\\
&\leq   n^\frac{\delta}{2}\sum\limits_{k=1}^nE|\mathfrak{D}_kF|^{2+\delta}\,\,\,\text{by Proposition \ref{77h4}, $(ii)$}.
\end{align*}
Hence, the point $(i)$ of Proposition \ref{77h4} tells us that the condition $F\in L^{2+\delta}(P)$ implies $E|\mathfrak{D}_kF|^{2+\delta}<\infty$ and $E|D_kZ^{(\alpha)}|^\frac{2+\delta}{2}\leq 2^\frac{2+\delta}{2}E|Z^{(\alpha)}|^\frac{2+\delta}{2}<\infty.$ So our bounds (\ref{8je4a}) and (\ref{8je4ab}) are well defined. Similarly, the condition $F\in L^{3+\delta}(P)$ ensures that the bounds (\ref{kofd4}) and (\ref{kofd44}) are well defined.
\end{rem}
\begin{rem}In the statement of Theorems \ref{lesser4} and \ref{ltser4a}, we introduced three new  parameters $\alpha,\beta$ and $\gamma.$ Let us give here an example to show the role of those parameters. Consider the sequence
$$F_n:=(X_1+...+X_{n-1})X_n,\,\,n\geq 2,$$
where $X_1,X_2, . . . , $ be the independent and identically distributed random variables with $P(X_i = 1) = P(X_i =-1) = 1/2,\,\,i\geq 1.$
Note that $E[F_n]=0,$ $\sigma_n^2:=E[F_n^2]=n-1,$ $E[F_n^3]=0$ and $E[F_n^4]<\infty.$ Hence, in both Theorems, the moment condition is satisfied with $\delta=1$ and we expect to obtain the optimal rates of convergence $O(n^{-1/2})$ and $O(n^{-1})$ for the distances $d_1$ and $d_2,$ respectively.

We have
$$\mathfrak{D}_kF_n=X_kX_n,\,\,E[\mathfrak{D}_kF_n|\mathcal{F}_k]=0,\,\,E[\mathfrak{D}_kF_n|\mathcal{G}_k]=X_kX_n,\,\,\,1\leq k\leq n-1,$$
$$\mathfrak{D}_nF_n=(X_1+...+X_{n-1})X_n,\,\,E[\mathfrak{D}_nF_n|\mathcal{F}_n]=\mathfrak{D}_nF_n,\,\,E[\mathfrak{D}_nF_n|\mathcal{G}_n]=0,$$
and hence,
$$Z^{(\alpha)}=\alpha(X_1+...+X_{n-1})^2+(1-\alpha)(n-1).$$
\noindent{\it The choice $\alpha=1$.} The bound (\ref{8je4a}) with $\delta=1$ becomes
\begin{align*}
&d_1(F_n/{\sigma_n},N)\leq \frac{2}{(n-1)^{3/2}} \sum\limits_{k=1}^nE\left[(|\mathfrak{D}_kF|+E_k|\mathfrak{D}_kF|)|\mathfrak{D}^{(\beta)}_kZ^{(1)}|\right]\notag\\
&\hspace{3cm}+\frac{8}{(n-1)^{3/2}} E|X_1+...+X_{n-1}|^3\,\,\forall\,\,\beta\in[0,1].
\end{align*}
So this choice fails to prove the central limit theorem for $F_n/{\sigma_n}$ because $\frac{8}{(n-1)^{3/2}} E|X_1+...+X_{n-1}|^3\nrightarrow 0$ as $n\to \infty.$

\noindent{\it The choice $\alpha=0$.} We have $Z^{(0)}=n-1$ and $\mathfrak{D}^{(\beta)}_kZ^{(0)}=0\,\,\forall\,\,1\leq k\leq n$ and $\beta\in[0,1].$ Now the bound (\ref{8je4a}) with $\delta=1$ will yield the optimal rate for the Wasserstein distance. Indeed, we have
\begin{align*}
d_1(F_n/{\sigma_n},N)&\leq \frac{4}{(n-1)^{3/2}} \sum\limits_{k=1}^{n-1} E[(|X_kX_n|^{2}+E_k|X_kX_n|^{2})|X_kX_n|]\\
&=\frac{8}{\sqrt{n-1}}.
\end{align*}
Furthermore, we have $Z^{(0,\beta)}=-F_n\,\,\forall\,\,\beta\in[0,1].$ Hence, by choosing $\gamma=0,$ the bound (\ref{kofd4}) with $\delta=1$ gives us
\begin{align*}
d_2(F_n/\sigma_n,N)&\leq  \frac{4}{(n-1)^2} \sum\limits_{k=1}^{n-1}E\left[(|X_kX_n|+E_k|X_kX_n|)|X_kX_n|\right]\\
&+\frac{8}{(n-1)^2}\sum\limits_{k=1}^{n-1}E\left[(|X_kX_n|^{3}+E_k|X_kX_n|^{3})|X_kX_n|\right]\\
&\leq \frac{24}{n-1}.
\end{align*}
The reader can verify that the choice $\gamma=1$ will fail to give the above optimal rate.
\end{rem}


\section{Examples}\label{hj48y} In this section, we provide some examples to illustrate the applicability of our abstract results. Although our examples are fundamental ones, to the best of our knowledge, the results of this section are new (except the bound (\ref{thuthuy01}) which was already obtained in \cite{Bikjalis1966,Ibragimov1966}).
\subsection{Partial sums}
Let $X_1,...,X_{n}$ are independent $\mathbb{R}$-valued random variables with $\mu_k:=E[X_k],v_k^2=Var(X_k)$ and $E|X_k|^{2+\delta}<\infty$ for some $\delta\in (0,1].$ Define $S_n:=\sum\limits_{k=1}^n(X_k-\mu_k)$ and $\sigma^2_n:=\sum\limits_{k=1}^nv_k^2.$

We have
$$\mathfrak{D}_kS_n=X_k-\mu_k=E[\mathfrak{D}_kS_n|\mathcal{F}_k]=E[\mathfrak{D}_kS_n|\mathcal{G}_k],\,\,\,1\leq k\leq n.$$
Hence, for any $\alpha\in [0,1],$
$$Z^{(\alpha)}=\sum\limits_{k=1}^n\mathfrak{D}_kS_n\mathfrak{D}_k^{(\alpha)}S_n=\sum\limits_{k=1}^n (X_k-\mu_k)^2$$
and, for any $\beta\in [0,1],$
$$\mathfrak{D}_kZ^{(\alpha)}=\mathfrak{D}^{(\beta)}_kZ^{(\alpha)}=(X_k-\mu_k)^2-v_k^2,\,\,\,1\leq k\leq n.$$
It is easy to see that
$$E|\mathfrak{D}_kZ^{(\alpha)}|^{\frac{2+\delta}{2}}\leq 2^{\frac{\delta}{2}}(E|X_k-\mu_k|^{2+\delta}+(E|X_k-\mu_k|^2)^{\frac{2+\delta}{2}})\leq 2^{\frac{2+\delta}{2}}E|X_k-\mu_k|^{2+\delta}.$$
So our bound (\ref{8je4ab}) yields
\begin{equation}\label{thuthuy01}
d_1(S_n/\sigma_n,N)\leq \frac{8+2^{2+\delta}}{\sigma_n^{2+\delta}}\sum\limits_{k=1}^nE|X_k-\mu_k|^{2+\delta},
\end{equation}
which recovers the classical bound (\ref{iif}) with $C_\delta=8+2^{2+\delta}.$

Let us now investigate the vanishing third moment phenomenon for $S_n.$
\begin{prop}Suppose that $\sum\limits_{k=1}^nE[(X_k-\mu_k)^3]=0$  and $E|X_k|^{3+\delta}<\infty$ for $1\leq k\leq n$ and for some $\delta\in (0,1].$ Then, we have
\begin{align}
d_2(S_n/\sigma_n,N)&\leq \frac{C_\delta}{\sigma_n^{3+\delta}} \sum\limits_{k=1}^nE|X_k-\mu_k|^{3+\delta},\label{taggg8}
\end{align}
where $C_\delta:=8+2^{2+\delta}\big(3^{\frac{\delta}{3}}(2+3^{\frac{3+\delta}{3}})\big)^{\frac{3}{3+\delta}}+2^{4+\delta}.$
\end{prop}
\begin{proof}We have
\begin{align*}
Z^{(\alpha,\beta)}&=\sum\limits_{k=1}^n((X_k-\mu_k)^2-v_k^2)(X_k-\mu_k)-\frac{1}{2}\sum\limits_{k=1}^n((X_k-\mu_k)^2+v_k^2)(X_k-\mu_k)\\
&=\frac{1}{2}\sum\limits_{k=1}^n((X_k-\mu_k)^3-3v_k^2(X_k-\mu_k)),
\end{align*}
and hence,
$$\mathfrak{D}_kZ^{(\alpha,\beta)}=\frac{(X_k-\mu_k)^3-3v_k^2(X_k-\mu_k)-E[(X_k-\mu_k)^3]}{2},\,\,1\leq k\leq n.$$
So we can get
\begin{align*}
\sum\limits_{k=1}^nE|\mathfrak{D}_kZ^{(\alpha,\beta)}|^{\frac{3+\delta}{3}}
&=\frac{1}{2^{\frac{3+\delta}{3}}}\sum\limits_{k=1}^nE|(X_k-\mu_k)^3-3v_k^2(X_k-\mu_k)-E[(X_k-\mu_k)^3]|^{\frac{3+\delta}{3}}\\
&\leq \frac{3^{\frac{\delta}{3}}(2+3^{\frac{3+\delta}{3}})}{2^{\frac{3+\delta}{3}}}\sum\limits_{k=1}^nE|X_k-\mu_k|^{3+\delta}.
\end{align*}
We also have
\begin{align*}
\sum\limits_{k=1}^nE|\mathfrak{D}_kZ^{(\alpha)}|^{\frac{3+\delta}{2}}&=\sum\limits_{k=1}^nE|(X_k-\mu_k)^2-v_k^2|^{\frac{3+\delta}{2}}\leq 2^{\frac{3+\delta}{2}}\sum\limits_{k=1}^nE|X_k-\mu_k|^{3+\delta}.
\end{align*}
Since $E[F^3]=2E[Z^{(\alpha,\beta)}]=\sum\limits_{k=1}^nE[(X_k-\mu_k)^3]=0,$ this allows us to use the bound (\ref{kofd44}) and we obtain (\ref{taggg8}).
\end{proof}
Clearly, when the random variables $X_k's$ have the same distribution with mean zero and variance one, the bound (\ref{taggg8}) becomes $d_2(\bar{S}_n,N)\leq C_\delta E|X_1|^{3+\delta} n^{-\frac{1+\delta}{2}}.$ This is a generalization of Proposition \ref{jjggk4} because $d_{\mathcal{H}_4}\leq d_2.$

\subsection{A sum of dependent random variables}
Fix an integer number $m\geq 1$ and let $X_1,...,X_{n+m-1}$ be independent random variables taking values in $\mathcal{X}.$ Let $\xi_i:\mathcal{X}^m\to \mathbb{R},1\leq i\leq n$ are measurable functions. In this section, we generalize the classical Lyapunov bound (\ref{iif}) to the following sum of dependent random variables
\begin{equation}\label{hh3h}
F=\sum\limits_{i=1}^n(\xi_i-E[\xi_i]),
\end{equation}
where $\xi_i:=\xi_i(X_i,...,X_{i+m-1}),1\leq i\leq n.$ We note that the run and scan statistics are two important  examples of the form (\ref{hh3h}), see e.g. \cite{Balakrishnan2002}.
\begin{prop}\label{fjer}We consider the nonlinear statistic $F$ defined by (\ref{hh3h}).

\noindent{\bf I.} Assume that $E|\xi_i|^{2+\delta}<\infty$ for $1\leq i\leq n$ and for some $\delta\in (0,1].$ Then, we have
\begin{align}
d_1(F/\sigma,N)
&\leq \frac{c_{m,\delta}}{\sigma^{2+\delta}}\sum\limits_{k=1}^nE|\xi_k-E[\xi_k]|^{2+\delta}.\label{oo2}
\end{align}
where $c_{m,\delta}:=(2m)^{2+\delta}\big(8(2m-1)+2^{2+\delta}\big).$

\noindent{\bf II.} Assume that $E[F^3]=0$ and $E|\xi_i|^{3+\delta}<\infty$ for $1\leq i\leq n$ and for some $\delta\in (0,1].$ Then, we have
\begin{equation}\label{kfl5}
d_2(F/\sigma,N)\leq \frac{C_{m,\delta}}{\sigma^{3+\delta}}\sum\limits_{k=1}^nE|\xi_k-E[\xi_k]|^{3+\delta},
\end{equation}
where $C_{m,\delta}:=(2m)^{3+\delta}\big(16(4m-2)^2+2^{3+\delta}(4m-2)+2^{3+\delta}\big).$
\end{prop}

\begin{proof}\noindent{\bf I.} In Theorem \ref{lesser4}, we choose to use $\alpha=1.$  Then, the bound (\ref{8je4ab}) gives us
\begin{align}d_1(F/\sigma,N)&\leq\frac{2^{2+\delta}}{\sigma^{2+\delta}} \sum\limits_{k=1}^{n+m-1} E|\mathfrak{D}_kF|^{2+\delta}\notag\\
&+ \frac{4}{\sigma^{2+\delta}} \bigg(\sum\limits_{k=1}^{n+m-1}E|\mathfrak{D}_kF|^{2+\delta}\bigg)^{\frac{\delta}{2+\delta}}
\bigg(\sum\limits_{k=1}^{n+m-1}E|\mathfrak{D}_kZ^{(\alpha)}|^{\frac{2+\delta}{2}}\bigg)^{\frac{2}{2+\delta}},\label{j7l1}
\end{align}
where $Z^{(\alpha)}:=\sum\limits_{l=1}^{n+m-1} \mathfrak{D}_lFE[\mathfrak{D}_lF|\mathcal{F}_l].$

We observe that $\mathfrak{D}_k\xi_i=0$ if $\xi_i$ does not depends on $X_k.$ Hence, using the convention $\xi_i=0$ if $i\leq 0$ or $i\geq n+1,$ we obtain
\begin{equation}\label{ddf4}
\mathfrak{D}_kF=\sum\limits_{i=k-m+1}^k\mathfrak{D}_k\xi_i,\,\,1\leq k\leq n+m-1.
\end{equation}
Then, we can get
$$\mathfrak{D}_kZ^{(\alpha)}=\sum\limits_{l=k-m+1}^{k+m-1} \mathfrak{D}_k(\mathfrak{D}_lFE[\mathfrak{D}_lF|\mathcal{F}_l]),\,\,1\leq k\leq n+m-1$$
with the convention $\mathfrak{D}_lF=0$ if $l\leq 0.$ By the fundamental inequality $(|a_1|+...+|a_N|)^m\leq N^{m-1}(|a_1|^m+...+|a_N|^m)$ and the point $(i)$ of Proposition \ref{77h4}, we deduce
\begin{align*}
E|D_kZ^{(\alpha)}|^{\frac{2+\delta}{2}}&\leq (2m-1)^{\frac{\delta}{2}}\sum\limits_{l=k-m+1}^{k+m-1} E|\mathfrak{D}_k(\mathfrak{D}_lFE[\mathfrak{D}_lF|\mathcal{F}_l])|^{\frac{2+\delta}{2}}\\
&\leq (2m-1)^{\frac{\delta}{2}}\sum\limits_{l=k-m+1}^{k+m-1} 2^{\frac{2+\delta}{2}}E|\mathfrak{D}_lFE[\mathfrak{D}_lF|\mathcal{F}_l]|^{\frac{2+\delta}{2}},\,\,1\leq k\leq n+m-1.
\end{align*}
Consequently,
$$\sum\limits_{k=1}^{n+m-1}E|D_kZ^{(\alpha)}|^{\frac{2+\delta}{2}} \leq (2m-1)^{1+\frac{\delta}{2}}\sum\limits_{k=1}^{n+m-1} 2^{\frac{2+\delta}{2}}E|\mathfrak{D}_kFE[\mathfrak{D}_kF|\mathcal{F}_k]|^{\frac{2+\delta}{2}}.$$
By using the Lyapunov and H\"older inequalities, we get $E|\mathfrak{D}_kFE[\mathfrak{D}_kF|\mathcal{F}_k]|^{\frac{2+\delta}{2}}\leq E|\mathfrak{D}_kF|^{2+\delta}.$ So it holds that
\begin{equation}\label{7j4r}
\sum\limits_{k=1}^{n+m-1}E|D_kZ^{(\alpha)}|^{\frac{2+\delta}{2}}\leq (4m-2)^{\frac{2+\delta}{2}}\sum\limits_{k=1}^{n+m-1} E|\mathfrak{D}_kF|^{2+\delta}.
\end{equation}
Inserting this relation into (\ref{j7l1}) yields
\begin{align*}
d_1(F/\sigma,N)&\leq \frac{8(2m-1)+2^{2+\delta}}{\sigma^{2+\delta}} \sum\limits_{k=1}^n E|\mathfrak{D}_kF|^{2+\delta}.
\end{align*}
Furthermore, from (\ref{ddf4}) and the point $(i)$ of Proposition \ref{77h4}, we deduce
\begin{align*}
E|\mathfrak{D}_kF|^{2+\delta}&\leq m^{1+\delta}\sum\limits_{i=k-m+1}^kE|\mathfrak{D}_k\xi_i|^{2+\delta}=m^{1+\delta}\sum\limits_{i=k-m+1}^kE|\mathfrak{D}_k(\xi_i-E[\xi_i])|^{2+\delta}\\
&\leq 2^{2+\delta}m^{1+\delta}\sum\limits_{i=k-m+1}^kE|\xi_i-E[\xi_i]|^{2+\delta},\,\,1\leq k\leq n+m-1,
\end{align*}
and hence,
$$\sum\limits_{k=1}^{n+m-1}E|\mathfrak{D}_kF|^{2+\delta}\leq 2^{2+\delta}m^{2+\delta}\sum\limits_{k=1}^{n}E|\xi_k-E[\xi_k]|^{2+\delta}.$$
So (\ref{oo2}) follows.

\noindent{\bf II.} Choosing $\alpha=\beta=1,$ the bound (\ref{kofd44}) gives us
\begin{align}
&d_2(F/\sigma,N)\leq \frac{8}{\sigma^{3+\delta}}\big(\sum\limits_{k=1}^{n+m-1}E|\mathfrak{D}_kF|^{3+\delta}\big)^{\frac{\delta}{3+\delta}}
\big(\sum\limits_{k=1}^{n+m-1}E|\mathfrak{D}_kZ^{(\alpha,\beta)}|^{\frac{3+\delta}{3}}\big)^{\frac{3}{3+\delta}}\notag\\
&+\frac{2^{3+\delta}}{\sigma^{3+\delta}}\big(\sum\limits_{k=1}^{n+m-1}E|\mathfrak{D}_kF|^{3+\delta}\big)^{\frac{1+\delta}{3+\delta}}
\big(\sum\limits_{k=1}^{n+m-1}E|\mathfrak{D}_kZ^{(\alpha)}|^{\frac{3+\delta}{2}}\big)^{\frac{2}{3+\delta}}
+\frac{2^{3+\delta}}{\sigma^{3+\delta}}\sum\limits_{k=1}^{n+m-1}E|\mathfrak{D}_kF|^{3+\delta},\label{j1qpe}
\end{align}
where $Z^{(\alpha)}$ is as in the part {\bf I} and
$$Z^{(\alpha,\beta)}:=\sum\limits_{k=1}^n\mathfrak{D}_kFE[\mathfrak{D}_kZ^{(\alpha)}|\mathcal{F}_k]-\frac{1}{2}\sum\limits_{k=1}^n(|\mathfrak{D}_kF|^2
+E_k|\mathfrak{D}_kF|^2)E[\mathfrak{D}_kF|\mathcal{F}_k].$$
Using the same arguments as in the proof of (\ref{7j4r}) we obtain
\begin{equation}\label{7j4rr}\sum\limits_{k=1}^{n+m-1}E|D_kZ^{(\alpha)}|^{\frac{3+\delta}{2}} \leq (4m-2)^{\frac{3+\delta}{2}}\sum\limits_{k=1}^{n+m-1} E|\mathfrak{D}_kF|^{3+\delta}.
\end{equation}
On the other hand, we have, for $1\leq k\leq n+m-1,$
$$\mathfrak{D}_kZ^{(\alpha,\beta)}=\sum\limits_{l=k-m+1}^{k+m-1} \mathfrak{D}_k\left(\mathfrak{D}_lFE[\mathfrak{D}_lZ^{(\alpha)}|\mathcal{F}_l]\right)-\frac{1}{2}\sum\limits_{l=k-m+1}^{k+m-1}\mathfrak{D}_k\left( (|\mathfrak{D}_lF|^2
+E_l|\mathfrak{D}_lF|^2)E[\mathfrak{D}_lF|\mathcal{F}_l]\right).$$
So we can deduce
\begin{align*}
E|\mathfrak{D}_kZ^{(\alpha,\beta)}|^{\frac{3+\delta}{3}}&\leq (4m-2)^{\frac{\delta}{3}}2^{\frac{3+\delta}{3}}\bigg(\sum\limits_{l=k-m+1}^{k+m-1} E|\mathfrak{D}_lFE[\mathfrak{D}_lZ^{(\alpha)}|\mathcal{F}_l]|^{\frac{3+\delta}{3}}\\
&+\frac{1}{2^{\frac{3+\delta}{3}}}\sum\limits_{l=k-m+1}^{k+m-1}E|(|\mathfrak{D}_lF|^2
+E_l|\mathfrak{D}_lF|^2)E[\mathfrak{D}_lF|\mathcal{F}_l]|^{\frac{3+\delta}{3}}\bigg)\\
&\leq (4m-2)^{\frac{\delta}{3}}\bigg(2^{\frac{3+\delta}{3}}\sum\limits_{l=k-m+1}^{k+m-1} E|\mathfrak{D}_lFE[\mathfrak{D}_lZ^{(\alpha)}|\mathcal{F}_l]|^{\frac{3+\delta}{3}}+\sum\limits_{l=k-m+1}^{k+m-1}E|\mathfrak{D}_lF|^{3+\delta}\bigg).
\end{align*}
We now observe that
\begin{align*}E|\mathfrak{D}_lFE[\mathfrak{D}_lZ^{(\alpha)}|\mathcal{F}_l]|^{\frac{3+\delta}{3}}&\leq \frac{E|\mathfrak{D}_lF|^{3+\delta}+2E|E[\mathfrak{D}_lZ^{(\alpha)}|\mathcal{F}_l]|^{\frac{3+\delta}{2}}}{3}\\
&\leq \frac{E|\mathfrak{D}_lF|^{3+\delta}+2E|\mathfrak{D}_lZ^{(\alpha)}|^{\frac{3+\delta}{2}}}{3},
\end{align*}
which, in turn, implies that
\begin{align*}
E|\mathfrak{D}_kZ^{(\alpha,\beta)}|^{\frac{3+\delta}{3}}&\leq (4m-2)^{\frac{\delta}{3}}2^{\frac{3+\delta}{3}}\bigg(\sum\limits_{l=k-m+1}^{k+m-1} E|\mathfrak{D}_lZ^{(\alpha)}|^{\frac{3+\delta}{2}}+\sum\limits_{l=k-m+1}^{k+m-1}E|\mathfrak{D}_lF|^{3+\delta}\bigg),
\end{align*}
We therefore obtain
\begin{align}
&\sum\limits_{k=1}^{n+m-1}E|\mathfrak{D}_kZ^{(\alpha,\beta)}|^{\frac{3+\delta}{3}}\notag\\
&\leq (4m-2)^{\frac{\delta}{3}}2^{\frac{3+\delta}{3}}\bigg((2m-1)\sum\limits_{k=1}^{n+m-1} E|\mathfrak{D}_kZ^{(\alpha)}|^{\frac{3+\delta}{2}}+(2m-1)\sum\limits_{k=1}^{n+m-1}E|\mathfrak{D}_kF|^{3+\delta}\bigg)\notag\\
&\leq \frac{1}{2}(8m-4)^{\frac{3+\delta}{3}}\big((4m-2)^{\frac{3+\delta}{2}}+1\big)
\sum\limits_{k=1}^{n+m-1}E|\mathfrak{D}_kF|^{3+\delta}\notag\\
&\leq (8m-4)^{\frac{3+\delta}{3}}(4m-2)^{\frac{3+\delta}{2}}
\sum\limits_{k=1}^{n+m-1}E|\mathfrak{D}_kF|^{3+\delta}.\label{j1qper}
\end{align}
Inserting (\ref{7j4rr}) and (\ref{j1qper}) into (\ref{j1qpe}) we obtain the bound (\ref{kfl5}) because
$$\sum\limits_{k=1}^{n+m-1}E|\mathfrak{D}_kF|^{3+\delta}\leq 2^{3+\delta}m^{3+\delta}\sum\limits_{k=1}^{n}E|\xi_k-E[\xi_k]|^{3+\delta}.$$
The proof of Proposition is complete.
\end{proof}
\begin{rem} In the proof of Proposition \ref{fjer}, we used $\alpha=\beta=1.$ The reader can verify that the other choices of $\alpha$ and $\beta$ give us similar bounds to (\ref{oo2}) and (\ref{kfl5}), the value of constants $c_{m,\delta}$ and $C_{m,\delta}$ may vary.
\end{rem}

\begin{exam}Let $X,X_1,X_2,...$ be independent and identically distributed $\mathbb{R}$-valued random variables with zero mean and unit variance. We consider the sequence of $m$-runs defined by
$$F_n:=n^{-1/2}\sum\limits_{i=1}^nX_i...X_{i+m-1},\,\,n\geq 1.$$
It is easy to see that $E[F_n]=0$ and $Var(F_n)=1.$ Hence, if $E|X|^{2+\delta}<\infty,$  then the bound (\ref{oo2}) gives us
$$d_1(F_n,N)\leq c_{m,\delta}(E|X|^{2+\delta})^mn^{-\frac{\delta}{2}}.$$
If $E[X^3]=0$ and $E|X|^{3+\delta}<\infty,$ then the bound (\ref{kfl5}) gives us
$$d_2(F_n,N)\leq C_{m,\delta}(E|X|^{3+\delta})^mn^{-\frac{1+\delta}{2}}.$$
\end{exam}

\subsection{CLT for quadratic forms without finite fourth moment}
Let $X_1,X_2,...$ be independent $\mathbb{R}$-valued random variables with zero means, unit variances and $A=(a^{(n)}_{uv})_{u,v=1}^n$  be a symmetric matrix with vanishing diagonal, where each $a^{(n)}_{uv}$ is a real number depending on $n.$ For the simplicity of notations, we will write $a_{uv}$ instead of $a^{(n)}_{uv}.$ The central limit theorem and normal approximation results for the quadratic form
$$F_n=\sum\limits_{1\leq u\leq v\leq n} a_{uv}X_uX_v$$
has been extensively discussed in the literature. The most of works require the finite fourth moment condition, i.e. $E|X_k|^4<\infty,k\geq 1.$ The best known result proved by de Jong \cite{deJong1987} tells us that the $F_n/\sigma_n$ converges to a standard normal random variable in distribution if
$$\text{$\sigma_n^{-4}\mathrm{Tr}(A^4)\to 0$ and $\tilde{L}_n:=\sigma_n^{-2}\max\limits_{1\leq u\leq n}\sum\limits_{v=1}^na_{uv}^2\to 0$},$$
where $\sigma^2_n:=Var(F_n)=\sum\limits_{1\leq u\leq v\leq n} a_{uv}^2$ and $\mathrm{Tr}(A^4)=\sum\limits_{u,v=1}^n \big(\sum\limits_{k=1}^na_{ku} a_{kv}\big)^2.$ Also see \cite{Dobler2017b,Shao2019} for the rates of convergence obtained there.

Here, in the next Proposition, we only require the random variables $X_k's$ to have the finite absolute moment of order $2+\delta.$ This is a significant supplement to the literature.
\begin{prop}\label{kfm4}
Assume that
\begin{equation}\label{o2a}
\lim\limits_{n\to \infty}\tilde{L}_n=0
\end{equation}
and
\begin{equation}\label{o1a}
\sup\limits_{n\geq 1}\frac{\tilde{L}_n^{\frac{\delta(\delta-\varepsilon)}{4}}}{\sigma_n^{2+\delta}}\left(\sum\limits_{u,v=1}^n\big|\sum\limits_{k=1}^na_{ku} a_{kv}\big|^{\frac{2+\delta}{2}}+\sum\limits_{u,v=1}^n\sum\limits_{k=1}^n|a_{ku} a_{kv}|^{\frac{2+\delta}{2}}\right)<\infty,
\end{equation}
for some $\varepsilon\in(0,\delta].$ Then, $F_n/\sigma_n$ converges in distribution to a standard normal random variable as $n\to\infty.$ Moreover, we have
\begin{align}
&d_1(F_n/\sigma_n,N)\leq\frac{C_\delta}{\sigma_n^{2+\delta}}  \big(\max\limits_{1\leq w\leq n}E|X_w|^{2+\delta}\big)^2\notag\\
&\times\bigg\{\bigg(\sum\limits_{u=1}^n\big(\sum\limits_{v=1}^n a^2_{uv}\big)^{\frac{2+\delta}{2}}\bigg)^{\frac{\delta}{2+\delta}}\bigg(\sum\limits_{u,v=1}^n\big|\sum\limits_{k=1}^na_{ku} a_{kv}\big|^{\frac{2+\delta}{2}}+\sum\limits_{u,v=1}^n\sum\limits_{k=1}^n |a_{ku} a_{kv}|^{\frac{2+\delta}{2}}\bigg)^{\frac{2}{2+\delta}}\notag\\
&\hspace{9cm}+\sum\limits_{u=1}^n\big(\sum\limits_{v=1}^n a^2_{uv}\big)^{\frac{2+\delta}{2}}\bigg\},\label{quad1}
\end{align}
where $C_\delta$ is a positive constant depending only on $\delta.$
\end{prop}
\begin{proof}We first use the bound (\ref{8je4ab}) to prove (\ref{quad1}). We have
$$\mathfrak{D}_kF_n=X_k\sum\limits_{u=1}^n a_{ku}X_u,\,\,1\leq k\leq n.$$
We choose to use $\alpha=\frac{1}{2}.$ Then, we obtain
$$\mathfrak{D}^{(\alpha)}_kF_n=\frac{1}{2}X_k\sum\limits_{u=1}^n a_{ku}X_u,\,\,1\leq k\leq n$$
and
$$Z^{(\alpha)}=\frac{1}{2}\sum\limits_{k=1}^nX^2_k\big(\sum\limits_{v=1}^n a_{kv}X_v\big)^2=\frac{1}{2}\sum\limits_{k=1}^nZ^{(\alpha)}_k,$$
where
$$Z^{(\alpha)}_k:=X^2_k\big(\sum\limits_{u=1}^n a_{ku}X_u\big)^2,\,\,1\leq k\leq n.$$
We observe that $\mathfrak{D}_l Z^{(\alpha)}_k=(X^2_k-1)\big(\sum\limits_{u=1}^n a_{ku}X_u\big)^2$ if $l=k$ and for $l\neq k,$
\begin{align*}
\mathfrak{D}_l Z^{(\alpha)}_k&=X^2_k\bigg(a^{2}_{kl}(X^2_l-1)+2a_{kl}X_l\sum\limits_{v=1,v\neq l}^n a_{kv}X_v\bigg).
\end{align*}
Hence, we obtain
\begin{align*}
\mathfrak{D}_l Z^{(\alpha)}
&=\frac{1}{2}(X^2_l-1)\big(\sum\limits_{u=1}^n a_{lu}X_u\big)^2\\
&+\frac{1}{2}(X^2_l-1)\sum\limits_{k=1}^n a^{2}_{kl}X^2_k+X_l\sum\limits_{v=1,v\neq l}^n\big(\sum\limits_{k=1}^na_{kl} a_{kv}X^2_k\big)X_v,\,\,1\leq l\leq n.
\end{align*}
It follows from the fundamental inequality $(|a_1|+|a_2|+|a_3|)^m\leq 3^{m-1}(|a_1|^m+|a_2|+|a_3|^m)$ that
\begin{align}
E|\mathfrak{D}_l Z^{(\alpha)}|^{\frac{2+\delta}{2}}\leq \frac{3^{\frac{\delta}{2}}}{2^{\frac{2+\delta}{2}}}\bigg(E|X^2_l-&1|^{\frac{2+\delta}{2}}E\big|\sum\limits_{u=1}^n a_{lu}X_u\big|^{2+\delta}+ E|X^2_l-1|^{\frac{2+\delta}{2}}E\big|\sum\limits_{k=1}^n a^{2}_{kl}X^2_k\big|^{\frac{2+\delta}{2}}\notag\\
&+2^{\frac{2+\delta}{2}}E|X_l|^{\frac{2+\delta}{2}}E\big|\sum\limits_{v=1,v\neq l}^n\big(\sum\limits_{k=1}^na_{kl} a_{kv}X^2_k\big)X_v\big|^{\frac{2+\delta}{2}}\bigg).\label{bgj3}
\end{align}
By the inequalities (\ref{9hj4}) and (\ref{k1h5}) below we deduce
$$E\big|\sum\limits_{u=1}^n a_{lu}X_u\big|^{2+\delta}\leq (1+\delta)^{\frac{2+\delta}{2}}\max\limits_{1\leq u\leq n}E|X_u|^{2+\delta}\big(\sum\limits_{u=1}^n a^2_{lu}\big)^{\frac{2+\delta}{2}}$$
and
\begin{align*}
E\big|\sum\limits_{k=1}^n a^{2}_{kl}X^2_k\big|^{\frac{2+\delta}{2}}&\leq \big(\sum\limits_{k=1}^n a^2_{kl}\big)^{\frac{2+\delta}{2}}+4\sum\limits_{k=1}^n E|a^{2}_{kl}(X^2_k-1)|^{\frac{2+\delta}{2}}\\
&\leq \big(\sum\limits_{k=1}^n a^2_{kl}\big)^{\frac{2+\delta}{2}}+4\max\limits_{1\leq k\leq n}E|X^2_k-1|^{\frac{2+\delta}{2}}\sum\limits_{k=1}^n |a_{kl}|^{2+\delta}\\
&\leq 9\times 2^{\frac{\delta}{2}}\max\limits_{1\leq k\leq n}E|X_k|^{2+\delta}\big(\sum\limits_{k=1}^n a^2_{kl}\big)^{\frac{2+\delta}{2}}.
\end{align*}
To estimate the third addend in the right hand side of (\ref{bgj3}), we put
$$W:=\sum\limits_{v=1,v\neq l}^n\big(\sum\limits_{k=1}^na_{kl} a_{kv}X^2_k\big)X_v.$$
We have $E[W]=0$ and
$$\mathfrak{D}_wW=\big(\sum\limits_{k=1}^na_{kl} a_{kw}X^2_k\big)X_w+\sum\limits_{v=1,v\neq l,w}^na_{wl} a_{wv}(X^2_w-1)X_v,\,\,1\leq w\leq n.$$
Hence, by the inequality (\ref{k1h5}), we obtain
\begin{align*}
E\big|&\sum\limits_{v=1,v\neq l}^n\big(\sum\limits_{k=1}^na_{kl}a_{kv}X^2_k\big)X_v\big|^{\frac{2+\delta}{2}}=E|W|^{\frac{2+\delta}{2}}\leq4\sum\limits_{w=1}^nE|\mathfrak{D}_wW|^{\frac{2+\delta}{2}}\\
&\leq 2^{2+\frac{\delta}{2}}\sum\limits_{w=1}^nE|\big(\sum\limits_{k=1}^na_{kl} a_{kw}X^2_k\big)X_w|^{\frac{2+\delta}{2}}+2^{2+\frac{\delta}{2}}\sum\limits_{w=1}^nE\big|\sum\limits_{v=1,v\neq l,w}^na_{wl} a_{wv}(X^2_w-1)X_v\big|^{\frac{2+\delta}{2}}\\
&\leq 2^{2+\frac{\delta}{2}}\max\limits_{1\leq w\leq n}E|X_w|^{\frac{2+\delta}{2}}\sum\limits_{w=1}^nE\big|\sum\limits_{k=1}^na_{kl} a_{kw}X^2_k\big|^{\frac{2+\delta}{2}}\\
&+2^{2+\frac{\delta}{2}}\max\limits_{1\leq w\leq n}E|X^2_w-1|^{\frac{2+\delta}{2}}\sum\limits_{w=1}^nE\big|\sum\limits_{v=1,v\neq l,w}^na_{wl} a_{wv}X_v\big|^{\frac{2+\delta}{2}}.
\end{align*}
Once again, we use the inequality (\ref{k1h5}) to get
\begin{align*}
&2^{2+\frac{\delta}{2}}\max\limits_{1\leq w\leq n}E|X_w|^{\frac{2+\delta}{2}}\sum\limits_{w=1}^nE\big|\sum\limits_{k=1}^na_{kl} a_{kw}X^2_k\big|^{\frac{2+\delta}{2}}\\
&\leq 2^{2+\frac{\delta}{2}}\max\limits_{1\leq w\leq n}E|X_w|^{\frac{2+\delta}{2}}\bigg(\sum\limits_{w=1}^n\big|\sum\limits_{k=1}^na_{kl} a_{kw}\big|^{\frac{2+\delta}{2}}+4\sum\limits_{w=1}^n\sum\limits_{k=1}^nE\big|a_{kl} a_{kw}(X^2_k-1)\big|^{\frac{2+\delta}{2}}\bigg)\\
&\leq 2^{2+\frac{\delta}{2}}\max\limits_{1\leq w\leq n}E|X_w|^{\frac{2+\delta}{2}}\bigg(\sum\limits_{w=1}^n\big|\sum\limits_{k=1}^na_{kl} a_{kw}\big|^{\frac{2+\delta}{2}}+4\max\limits_{1\leq k\leq n}E|X_k^2-1|^{\frac{2+\delta}{2}}\sum\limits_{w=1}^n\sum\limits_{k=1}^n|a_{kl} a_{kw}|^{\frac{2+\delta}{2}}\bigg)
\end{align*}
and
\begin{align*}
&2^{2+\frac{\delta}{2}}\max\limits_{1\leq w\leq n}E|X^2_w-1|^{\frac{2+\delta}{2}}\sum\limits_{w=1}^nE\big|\sum\limits_{v=1,v\neq l,w}^na_{wl} a_{wv}X_v\big|^{\frac{2+\delta}{2}}\\
&\leq 2^{4+\frac{\delta}{2}}\max\limits_{1\leq w\leq n}E|X^2_w-1|^{\frac{2+\delta}{2}}\sum\limits_{w=1}^n\sum\limits_{v=1,v\neq l,w}^nE|a_{wl} a_{wv}X_v|^{\frac{2+\delta}{2}}\\
&\leq 2^{4+\frac{\delta}{2}}\max\limits_{1\leq w\leq n}E|X^2_w-1|^{\frac{2+\delta}{2}}\max\limits_{1\leq w\leq n}E|X_w|^{\frac{2+\delta}{2}}\sum\limits_{w=1}^n\sum\limits_{v=1,v\neq l,w}^n|a_{wl} a_{wv}|^{\frac{2+\delta}{2}}\\
&\leq 2^{4+\frac{\delta}{2}}\max\limits_{1\leq w\leq n}E|X^2_w-1|^{\frac{2+\delta}{2}}\max\limits_{1\leq w\leq n}E|X_w|^{\frac{2+\delta}{2}}\sum\limits_{v=1}^n\sum\limits_{w=1}^n |a_{wl} a_{wv}|^{\frac{2+\delta}{2}}.
\end{align*}
Combining the above computations, we obtain from (\ref{bgj3}) that
\begin{align*}
E|\mathfrak{D}_l Z^{(\alpha)}|^{\frac{2+\delta}{2}}&\leq C_\delta \bigg(E|X^2_l-1|^{\frac{2+\delta}{2}}\max\limits_{1\leq u\leq n}E|X_u|^{2+\delta}\big(\sum\limits_{u=1}^n a^2_{lu}\big)^{\frac{2+\delta}{2}}\\
&+E|X_l|^{\frac{2+\delta}{2}}\max\limits_{1\leq w\leq n}E|X_w|^{\frac{2+\delta}{2}}\sum\limits_{w=1}^n\big|\sum\limits_{k=1}^na_{kl} a_{kw}\big|^{\frac{2+\delta}{2}}\\
&+E|X_l|^{\frac{2+\delta}{2}}\max\limits_{1\leq w\leq n}E|X^2_w-1|^{\frac{2+\delta}{2}}\max\limits_{1\leq w\leq n}E|X_w|^{\frac{2+\delta}{2}}\sum\limits_{v=1}^n\sum\limits_{w=1}^n |a_{wl} a_{wv}|^{\frac{2+\delta}{2}}\bigg),
\end{align*}
where $C_\delta$ is a positive constant depending only on $\delta.$ Consequently, for some $C_\delta>0,$
\begin{align}
&\sum\limits_{k=1}^nE|\mathfrak{D}_k Z^{(\alpha)}|^{\frac{2+\delta}{2}}\leq C_\delta \bigg(\max\limits_{1\leq w\leq n}E|X_w|^{2+\delta}\sum\limits_{u,v=1}^n\big|\sum\limits_{k=1}^na_{ku} a_{kv}\big|^{\frac{2+\delta}{2}}\notag\\
&+\max\limits_{1\leq w\leq n}E|X^2_w-1|^{\frac{2+\delta}{2}}\max\limits_{1\leq w\leq n}E|X_w|^{2+\delta}\bigg(\sum\limits_{u,v=1}^n\sum\limits_{k=1}^n |a_{ku} a_{kv}|^{\frac{2+\delta}{2}}+\sum\limits_{u=1}^n\big(\sum\limits_{v=1}^n a^2_{uv}\big)^{\frac{2+\delta}{2}}\bigg)\bigg)\notag\\
&\leq C_\delta \big(\max\limits_{1\leq w\leq n}E|X_w|^{2+\delta}\big)^2\bigg(\sum\limits_{u,v=1}^n\big|\sum\limits_{k=1}^na_{ku} a_{kv}\big|^{\frac{2+\delta}{2}}+\sum\limits_{u,v=1}^n\sum\limits_{k=1}^n |a_{ku} a_{kv}|^{\frac{2+\delta}{2}}+\sum\limits_{u=1}^n\big(\sum\limits_{v=1}^n a^2_{uv}\big)^{\frac{2+\delta}{2}}\bigg).\label{9ja1}
\end{align}
On the other hand, we use the inequality (\ref{9hj4}) to get
\begin{align}
E|\mathfrak{D}_kF_n|^{2+\delta}&=E|X_k|^{2+\delta}E\big|\sum\limits_{u=1}^n a_{ku}X_u\big|^{2+\delta}\notag\\
&\leq (1+\delta)^{\frac{2+\delta}{2}}\big(\max\limits_{1\leq u\leq n}E|X_u|^{2+\delta}\big)^2\big(\sum\limits_{u=1}^n a^2_{ku}\big)^{\frac{2+\delta}{2}},\,\,1\leq k\leq n,\notag
\end{align}
and hence,
\begin{equation}\label{9ja2}
\sum\limits_{k=1}^nE|\mathfrak{D}_kF_n|^{2+\delta}\leq (1+\delta)^{\frac{2+\delta}{2}}\big(\max\limits_{1\leq u\leq n}E|X_u|^{2+\delta}\big)^2\sum\limits_{u=1}^n\big(\sum\limits_{v=1}^n a^2_{uv}\big)^{\frac{2+\delta}{2}}.
\end{equation}
Inserting the estimates (\ref{9ja1}) and (\ref{9ja2}) into (\ref{8je4ab}) gives us the bound (\ref{quad1}).

To finish the proof, we observe that
$$\frac{1}{\sigma_n^{2+\delta}}\sum\limits_{u=1}^n\big(\sum\limits_{v=1}^n a^2_{uv}\big)^{\frac{2+\delta}{2}}\leq \frac{1}{\sigma_n^{\delta}}\max\limits_{1\leq u\leq n}\big(\sum\limits_{v=1}^na_{uv}^2\big)^{\frac{\delta}{2}}=\tilde{L}^{\frac{\delta}{2}}_n.$$
Hence, the conditions (\ref{o2a}) and (\ref{o1a}) ensure that $d_1(F_n/\sigma_n,N)\to 0$ as $n\to \infty.$ So $F_n/\sigma_n$ converges in distribution to $N.$

The proof of Proposition is complete.
\end{proof}
\begin{rem}In the proof of Proposition \ref{kfm4}, we used $\alpha=\frac{1}{2}$ because we want to obtain similar conclusions as in \cite{deJong1987}. If we choose to use $\alpha=1,$ then the bound (\ref{quad1}) depends on $\sum\limits_{u,v=1}^n\big|\sum\limits_{k=u}^na_{ku} a_{kv}\big|^{\frac{2+\delta}{2}}$ instead of $\sum\limits_{u,v=1}^n\big|\sum\limits_{k=1}^na_{ku} a_{kv}\big|^{\frac{2+\delta}{2}}.$
\end{rem}

\section{Proofs of the main results}\label{8i2}
Our proof will repeatedly use the following covariance  formula, see Proposition 2.3 in \cite{dungnt2019}.
\begin{prop}\label{lods3} Let $U=U(X)$ and $V=V(X)$ be two random variables in $L^2(P).$ For any $\alpha\in[0,1],$ we have
\begin{equation}\label{jfmn4}
Cov(U,V)=E\left[\sum\limits_{i=1}^n \mathfrak{D}_i U \mathfrak{D}^{(\alpha)}_i V\right],
\end{equation}
where we recall that $\mathfrak{D}^{(\alpha)}_i V=\alpha E[\mathfrak{D}_iV|\mathcal{F}_i]+(1-\alpha)E[\mathfrak{D}_iV|\mathcal{G}_i].$
\end{prop}
Here, we note that the condition $U,V\in L^2(P)$ can be replaced by $U\in L^p(P)$ and $V\in L^q(P)$ for some $p,q>1$ with $\frac{1}{p}+\frac{1}{q}=1.$ In particular, if $U$ is bounded, we only need $V\in L^1(P).$ This is due to the fact that, under such conditions, all expectations in (\ref{jfmn4}) exist and hence, this formula still holds true. We also note that the formula (\ref{jfmn4}) can be seen as an extension of the covariance identity on page 7 of \cite{Privault2018}.

In the proof, we also use the following notations. We let $X'=(X'_1,X'_2,...,X'_n)$ be an independent copy of $X.$ Given a random variable $U=U(X),$ for each $1\leq k\leq n,$ we write $T_kU=U(X_1,...,X_{k-1},X'_k,X_{k+1},...,X_n)$  and denote by $E'_k$ the expectation with respect to $X'_k.$

\subsection{Proof of Theorem \ref{lesser4}}
As mentioned in Introduction, the key allowing us to relax moment conditions is simple observations about the solution of Stein's equation. We have the following.
\begin{prop}\label{tt1}Given an absolutely continuous function $h,$ we consider Stein's equation
\begin{equation}\label{fklls3}
f'(x)-xf(x)=h(x)-E[h(N)],\,\,x\in \mathbb{R}.
\end{equation}
There then exists a solution to the equation (\ref{fklls3}) that satisfies, for any $\delta\in (0,1],$
\begin{equation}\label{8uw}
|f'_h(x)-f'_h(y)|\leq 2\|h'\|_\infty|x-y|^\delta\,\,\,\forall\,\,x,y\in \mathbb{R}.
\end{equation}
\end{prop}
\begin{proof}It is known from Lemma 2.4 in \cite{Chen2011} that there exists a solution to the equation (\ref{fklls3}) that satisfies  $\|f'\|_\infty\leq \sqrt{2/\pi}\|h'\|_\infty$ and $\|f''\|_\infty\leq 2\|h'\|_\infty.$ Hence, if $|x-y|\leq 1,$ then
$$|f'(x)-f'(y)|\leq \|f''\|_\infty|x-y|\leq 2\|h'\|_\infty|x-y|\leq 2\|h'\|_\infty|x-y|^\delta.$$
If $|x-y|> 1,$ we have
$$|f'(x)-f'(y)|\leq 2\|f'\|_\infty\leq 2\sqrt{2/\pi}\|h'\|_\infty\leq 2\|h'\|_\infty|x-y|^\delta.$$
This finishes the proof.
\end{proof}

\noindent{\it Proof of Theorem \ref{lesser4}.} Without loss of generality, we can and will assume that $\sigma=1.$ Let $f$ be the solution to Stein's equation (\ref{fklls3}) as in Proposition \ref{tt1}. Then, the Wasserstein distance can be represented as follows
$$d_1(F,N)=\sup\limits_{h\in \mathcal{C}^1:\|h'\|_\infty\leq 1}|E[Ff(F)]-E[f'(F)]|.$$
We separate the proof into three steps.

\noindent{\it Step 1.} In this step, we claim that, for every $1\leq k\leq n,$
\begin{equation}\label{u7f}
\mathfrak{D}_kf(F)=f'(F)\mathfrak{D}_kF+R_{k},
\end{equation}
where $R_k$ is bounded by
\begin{equation}\label{vvb4a}
|R_k|\leq  2^{1+\delta}\|h'\|_\infty (|\mathfrak{D}_kF|^{1+\delta}+E_k|\mathfrak{D}_kF|^{1+\delta}).
\end{equation}
To prove (\ref{u7f}), we observe that $\mathfrak{D}_kF=E'_k[F-T_kF]$ and
$$\mathfrak{D}_kf(F)=f(F)-E_k[f(F)]=E'_k[f(F)-f(T_kF)].$$
Then, by the Lagrange theorem, there exists a random variable $\theta_k$ lying between $F$ and $T_kF$ such that
$$f(T_kF)-f(F)=f'(\theta_k)(T_kF-F)=f'(F)(T_kF-F)+\left(f'(\theta_k)-f'(F)\right)(T_kF-F).$$
Taking the expectation with respect to $X'_k$ we obtain (\ref{u7f}) with $R_k$ defined by
$$R_k:=-E'_k[\left(f'(\theta_k)-f'(F)\right)(T_kF-F)],\,\,\,1\leq k\leq n.$$
It follows from the estimate (\ref{8uw}) that
\begin{align*}
|R_k|&\leq 2\|h'\|_\infty E'_k[|\theta_k-F|^\delta|T_kF-F|]\\
&= 2\|h'\|_\infty E'_k|T_kF-F|^{1+\delta}\\
&\leq 2^{1+\delta}\|h'\|_\infty (E'_k|T_kF-E_kF|^{1+\delta}+|E_kF-F|^{1+\delta})\\
&=2^{1+\delta}\|h'\|_\infty (E_k|\mathfrak{D}_kF|^{1+\delta}+|\mathfrak{D}_kF|^{1+\delta}).
\end{align*}
So the claim (\ref{u7f}) is verified.

\noindent{\it Step 2.} We now use the covariance formula (\ref{jfmn4}) to get, for any $\alpha\in[0,1],$
\begin{align*}
E[Ff(F)]=Cov(f(F),F)&=E\left[\sum\limits_{k=1}^n\mathfrak{D}_kf(F)\mathfrak{D}_k^{(\alpha)}F\right]\\
&=E\left[\sum\limits_{k=1}^nf'(F)\mathfrak{D}_kF\mathfrak{D}_k^{(\alpha)}F\right]+E\left[\sum\limits_{k=1}^nR_k\mathfrak{D}_k^{(\alpha)}F\right]\\
&=E[f'(F)Z^{(\alpha)}]+E\left[\sum\limits_{k=1}^nR_k\mathfrak{D}_k^{(\alpha)}F\right].
\end{align*}
As a consequence,
\begin{align}\label{vvb4}
E[Ff(F)]-E[f'(F)]&=E[f'(F)(Z^{(\alpha)}-1)]+\sum\limits_{k=1}^nE[R_k\mathfrak{D}_k^{(\alpha)}F].
\end{align}
We note that $f'(F)$ is bounded, $E|Z^{(\alpha)}|$ finite and $E[Z^{(\alpha)}]=Cov(F,F)=1.$ Hence, once again, we can use the covariance formula (\ref{jfmn4}) to get
$$E[f'(F)(Z^{(\alpha)}-1)]=Cov(f'(F),Z^{(\alpha)})=E\bigg[\sum\limits_{k=1}^n\mathfrak{D}_kf'(F)\mathfrak{D}^{(\beta)}_kZ^{(\alpha)}\bigg]$$
for any $\beta\in[0,1].$  Since $\mathfrak{D}_kf'(F)=E'_k[f'(F)-f'(T_kF)],$ the estimate (\ref{8uw}) gives us
\begin{align*}
|\mathfrak{D}_kf'(F)|&\leq 2\|h'\|_\infty E'_k|F-T_kF|^\delta\\
&\leq 2\|h'\|_\infty (|F-E_k[F]|^\delta+E'_k|T_kF-E_k[F]|^\delta)\\
&=2\|h'\|_\infty (|\mathfrak{D}_kF|^{\delta}+E_k|\mathfrak{D}_kF|^{\delta}),\,\,1\leq k\leq n.
\end{align*}
We therefore obtain
$$|E[f'(F)(Z^{(\alpha)}-1)]|\leq 2\|h'\|_\infty E\bigg[\sum\limits_{k=1}^n(|\mathfrak{D}_kF|^{\delta}+E_k|\mathfrak{D}_kF|^{\delta})|D^{(\beta)}_kZ^{(\alpha)}|\bigg]$$
For the second addend in the right hand side of (\ref{vvb4}), recalling (\ref{vvb4a}), we obtain the following estimate
$$\big|\sum\limits_{k=1}^nE[R_k\mathfrak{D}_k^{(\alpha)}F]\big|\leq 2^{1+\delta}\|h'\|_\infty E\bigg[\sum\limits_{k=1}^n (|\mathfrak{D}_kF|^{1+\delta}+E_k|\mathfrak{D}_kF|^{1+\delta})|D_k^{(\alpha)}F|\bigg].$$
Thus we can conclude that
\begin{align*}
|E[Ff(F)]-E[f'(F)]|&\leq 2\|h'\|_\infty E\bigg[\sum\limits_{k=1}^n(|\mathfrak{D}_kF|^{\delta}+E_k|\mathfrak{D}_kF|^{\delta})|D^{(\beta)}_kZ^{(\alpha)}|\bigg]\\
&+2^{1+\delta}\|h'\|_\infty E\bigg[\sum\limits_{k=1}^n (|\mathfrak{D}_kF|^{1+\delta}+E_k|\mathfrak{D}_kF|^{1+\delta})|D_k^{(\alpha)}F|\bigg].
\end{align*}
Then taking the supremum over all $h$ satisfying $\|h'\|_\infty\leq 1$ yields
\begin{align*}
d_1(F,N)&\leq 2 \sum\limits_{k=1}^nE\left[(|\mathfrak{D}_kF|^{\delta}+E_k|\mathfrak{D}_kF|^{\delta})|D^{(\beta)}_kZ^{(\alpha)}|\right]\\
&+2^{1+\delta} \sum\limits_{k=1}^n E\left[[(|\mathfrak{D}_kF|^{1+\delta}+E_k|\mathfrak{D}_kF|^{1+\delta})|D_k^{(\alpha)}F|\right]
\end{align*}
and the bound (\ref{8je4a}) follows replacing $F$ by $F/\sigma.$

\noindent{\it Step 3.} In this step, we verify the bound (\ref{8je4ab}). We use the H\"older inequality and the point $(ii)$ of Proposition \ref{77h4} to get
\begin{align*}
E\left[|\mathfrak{D}_kF|^{\delta}|D^{(\beta)}_kZ^{(\alpha)}|\right]&\leq (E|\mathfrak{D}_kF|^{2+\delta})^{\frac{\delta}{2+\delta}}(E|D^{(\beta)}_kZ^{(\alpha)}|^{\frac{2+\delta}{2}})^{\frac{2}{2+\delta}}\\
&\leq (E|\mathfrak{D}_kF|^{2+\delta})^{\frac{\delta}{2+\delta}}(E|D_kZ^{(\alpha)}|^{\frac{2+\delta}{2}})^{\frac{2}{2+\delta}},\,\,1\leq k\leq n.
\end{align*}
On the other hand, by the independence, we have $E_k|\mathfrak{D}_kF|^{\delta}=E[|\mathfrak{D}_kF|^{\delta}|X_i,i\neq k].$ Then, by Lyaponov's inequality, we obtain $E|E_k|\mathfrak{D}_kF|^{\delta}|^{\frac{2+\delta}{2}}\leq E|\mathfrak{D}_kF|^{2+\delta},$ and hence, we also have
$$E\left[E_k|\mathfrak{D}_kF|^{\delta}|D^{(\beta)}_kZ^{(\alpha)}|\right]\leq (E|\mathfrak{D}_kF|^{2+\delta})^{\frac{\delta}{2+\delta}}(E|D_kZ^{(\alpha)}|^{\frac{2+\delta}{2}})^{\frac{2}{2+\delta}},\,\,1\leq k\leq n.$$
So it holds that
\begin{align}
&\sum\limits_{k=1}^nE\left[(|\mathfrak{D}_kF|^{\delta}+E_k|\mathfrak{D}_kF|^{\delta})|D^{(\beta)}_kZ^{(\alpha)}|\right]\notag\\
&\hspace{3cm}\leq 2 \sum\limits_{k=1}^n(E|\mathfrak{D}_kF|^{2+\delta})^{\frac{\delta}{2+\delta}}(E|D_kZ^{(\alpha)}|^{\frac{2+\delta}{2}})^{\frac{2}{2+\delta}}\notag\\
&\hspace{3cm}\leq 2 \big(\sum\limits_{k=1}^nE|\mathfrak{D}_kF|^{2+\delta}\big)^{\frac{\delta}{2+\delta}}
\big(\sum\limits_{k=1}^nE|D_kZ^{(\alpha)}|^{\frac{2+\delta}{2}}\big)^{\frac{2}{2+\delta}}.\label{8j3a}
\end{align}
With the same arguments above, we obtain
\begin{equation}\label{8j3ab}
\sum\limits_{k=1}^n E\left[(|\mathfrak{D}_kF|^{1+\delta}+E_k|\mathfrak{D}_kF|^{1+\delta})|D_k^{(\alpha)}F|\right]\leq 2 \sum\limits_{k=1}^n E|\mathfrak{D}_kF|^{2+\delta}
\end{equation}
Inserting (\ref{8j3a}) and (\ref{8j3ab}) into (\ref{8je4a}) we obtain the bound (\ref{8je4ab}).

The proof of Theorem \ref{lesser4} is complete.\hfill$\square$

\subsection{Proof of Theorem \ref{ltser4a}}
For the distance $d_2,$ we need the following observation about the solution of Stein's equation.
\begin{prop}\label{tt2}
Let $h\in \mathcal{C}^2$ with bounded derivatives. Then the equation (\ref{fklls3}) has a solution in $\mathcal{C}^3$ that satisfies, for any $\delta\in (0,1],$
\begin{equation}\label{8uwa}
|f''(x)-f''(y)|\leq 2(2\|h'\|_\infty\vee\|h''\|_\infty)|x-y|^\delta\,\,\,\forall\,\,x,y\in \mathbb{R}.
\end{equation}
\end{prop}
\begin{proof}It is known from Theorem 1.1 in \cite{Daly2008} that there exists a solution to the equation (\ref{fklls3}) that satisfies $\|f''\|_\infty\leq 2\|h'\|_\infty$ and $\|f'''\|_\infty\leq 2\|h''\|_\infty.$  Hence, the proof of (\ref{8uwa}) is similar to that of (\ref{8uw}). So we omit it.
\end{proof}
We also need a technical lemma.
\begin{lem}\label{fo02}Let $F=F(X)\in L^{3}(P)$ be centered and $Z^{(\alpha,\beta)}$ be as in Theorem \ref{ltser4a}.  It holds that
$$E[F^3]=2E[Z^{(\alpha,\beta)}].$$
\end{lem}
\begin{proof} It is easy to check that
$$\mathfrak{D}_kF^2=2F\mathfrak{D}_kF-|\mathfrak{D}_kF|^2-E_k|\mathfrak{D}_kF|^2,\,\,1\leq k\leq n.$$
Hence, by the covariance formula (\ref{jfmn4}), we obtain
\begin{align*}
E[F^3]&=E\big[\sum\limits_{k=1}^n \mathfrak{D}_kF^2\mathfrak{D}^{(\alpha)}_kF\big]\\
&=2E\big[F\sum\limits_{k=1}^n \mathfrak{D}_kF\mathfrak{D}^{(\alpha)}_kF\big]-E\big[\sum\limits_{k=1}^n (|\mathfrak{D}_kF|^2+E_k|\mathfrak{D}_kF|^2)\mathfrak{D}^{(\alpha)}_kF\big]\\
&=2E[FZ^{(\alpha)}]-E\big[\sum\limits_{k=1}^n (|\mathfrak{D}_kF|^2+E_k|\mathfrak{D}_kF|^2)\mathfrak{D}^{(\alpha)}_kF\big]\\
&=2E\big[\sum\limits_{k=1}^n \mathfrak{D}_kF\mathfrak{D}^{(\beta)}_kZ^{(\alpha)}\big]-E\big[\sum\limits_{k=1}^n (|\mathfrak{D}_kF|^2+E_k|\mathfrak{D}_kF|^2)\mathfrak{D}^{(\alpha)}_kF\big]\\
&=2E[Z^{(\alpha,\beta)}].
\end{align*}
This completes the proof.
\end{proof}
\noindent{\it Proof of Theorem \ref{ltser4a}.} It suffices to consider $\sigma=1.$ Let $f$ be a solution of the equation (\ref{fklls3}) that has the properties mentioned in Proposition \ref{tt2}. We have
$$d_2(F,N)=\sup\limits_{h\in \mathcal{C}^2:\|h'\|_\infty,\|h''\|_\infty\leq 1}|E[Ff(F)]-E[f'(F)]|.$$
We separate the proof into three steps.

\noindent{\it Step 1.} We claim that, for every $1\leq k\leq n,$
\begin{align}
\mathfrak{D}_kf(F)=f'(F)\mathfrak{D}_kF-\frac{1}{2}f''(F)(|\mathfrak{D}_kF|^2+E_k|\mathfrak{D}_kF|^2)+R_k\label{u7fc}
\end{align}
where $R_k$ is bounded by
\begin{equation}\label{v34a}
|R_k|\leq 2^{1+\delta}(2\|h'\|_\infty\vee\|h''\|_\infty)(|\mathfrak{D}_kF|^{2+\delta}+E_k|\mathfrak{D}_kF|^{2+\delta}).
\end{equation}
By the Taylor expansion, there exists a random variable $\theta_k$ lying between $F$ and $T_kF$ such that
\begin{align*}
f(T_kF)-f(F)&=f'(F)(T_kF-F)+\frac{1}{2}f''(\theta_k)(T_kF-F)^2\\
&=f'(F)(T_kF-F)+\frac{1}{2}f''(F)(T_kF-F)^2+\frac{1}{2}\left(f''(\theta_k)-f''(F)\right)(T_kF-F)^2.
\end{align*}
We observe that $E'_k[(T_kF-F)^2]=|\mathfrak{D}_kF|^2+E_k|\mathfrak{D}_kF|^2.$ Hence, by taking the expectation with respect to $X'_k,$ we obtain (\ref{u7fc}) with $R_k$ defined by
$$R_k:=-\frac{1}{2}E'_k[\left(f''(\theta_k)-f''(F)\right)(T_kF-F)^2],\,\,\,1\leq k\leq n.$$
Thanks to the estimate (\ref{8uwa}) we get
\begin{align*}
|R_k|&\leq (2\|h'\|_\infty\vee\|h''\|_\infty)E'_k[|\theta_k-F|^\delta|T_kF-F|^2]\\
&\leq (2\|h'\|_\infty\vee\|h''\|_\infty)E'_k|T_kF-F|^{2+\delta} \\
&\leq 2^{1+\delta}(2\|h'\|_\infty\vee\|h''\|_\infty)(|\mathfrak{D}_kF|^{2+\delta}+E_k|\mathfrak{D}_kF|^{2+\delta}).
\end{align*}
This completes the proof of (\ref{u7fc}).

\noindent{\it Step 2.} For any $\alpha\in [0,1],$ by the covariance formula (\ref{jfmn4}) and the result of the previous step, we deduce
\begin{align*}
&E[Ff(F)]=E\big[\sum\limits_{k=1}^n\mathfrak{D}_kf(F)\mathfrak{D}^{(\alpha)}_kF\big]\\
&=E\big[f'(F)\sum\limits_{k=1}^n\mathfrak{D}_kF\mathfrak{D}^{(\alpha)}_kF\big]
-\frac{1}{2}E\big[f''(F)\sum\limits_{k=1}^n(|\mathfrak{D}_kF|^2+E_k|\mathfrak{D}_kF|^2)\mathfrak{D}^{(\alpha)}_kF\big]
+\sum\limits_{k=1}^nE[R_k\mathfrak{D}^{(\alpha)}_kF]\\
&=E[f'(F)Z^{(\alpha)}]-E[f''(F)U^{(\alpha)}]+\sum\limits_{k=1}^nE[R_k\mathfrak{D}^{(\alpha)}_kF],
\end{align*}
where
$$U^{(\alpha)}:=\frac{1}{2}\sum\limits_{k=1}^n(|\mathfrak{D}_kF|^2+E_k|\mathfrak{D}_kF|^2)\mathfrak{D}^{(\alpha)}_kF.$$
Since $E[Z^{(\alpha)}]=Var(F)=1,$ we obtain
\begin{align}
E[Ff(F)]&-E[f'(F)]=E[f'(F)(Z^{(\alpha)}-1)]-E[f''(F)U^{(\alpha)}]+\sum\limits_{k=1}^nE[R_k\mathfrak{D}^{(\alpha)}_kF]\notag\\
&=E\big[\sum\limits_{k=1}^n\mathfrak{D}_kf'(F)\mathfrak{D}^{(\beta)}_kZ^{(\alpha)}\big]-E[f''(F)U^{(\alpha)}]
+\sum\limits_{k=1}^nE[R_k\mathfrak{D}^{(\alpha)}_kF]\label{kl3la}
\end{align}
for any $\beta\in [0,1].$

By using the same argument as in the proof of (\ref{u7f}) we have
\begin{equation}\label{kl3l}
\mathfrak{D}_kf'(F)=f''(F)\mathfrak{D}_kF+\bar{R}_k,\,\,1\leq k\leq n,
\end{equation}
where $\bar{R}_k:=-E'_k[\left(f''(\theta_k)-f''(F)\right)(T_kF-F)].$ Moreover, it follows from the estimate (\ref{8uwa}) that $\bar{R}_k$ is bounded by
\begin{align}
|\bar{R}_k|&\leq 2(2\|h'\|_\infty\vee\|h''\|_\infty)E'_k|T_kF-F|^{1+\delta}\notag\\
&\leq  2^{1+\delta}(2\|h'\|_\infty\vee\|h''\|_\infty) (|\mathfrak{D}_kF|^{1+\delta}+E_k|\mathfrak{D}_kF|^{1+\delta}),\,\,1\leq k\leq n.\label{v5b4a}
\end{align}
Inserting (\ref{kl3l}) into (\ref{kl3la}) yields
\begin{align}
&E[Ff(F)]-E[f'(F)]\notag\\
&=E\big[\sum\limits_{k=1}^nf''(F)\mathfrak{D}_kF\mathfrak{D}^{(\beta)}_kZ^{(\alpha)}\big]
+E\big[\sum\limits_{k=1}^n\bar{R}_k\mathfrak{D}^{(\beta)}_kZ^{(\alpha)}\big]-E[f''(F)U^{(\alpha)}]
+\sum\limits_{k=1}^nE[R_k\mathfrak{D}^{(\alpha)}_kF]\notag\\
&=E[f''(F)Z^{(\alpha,\beta)}]+\sum\limits_{k=1}^nE[R_k\mathfrak{D}^{(\alpha)}_kF]
+\sum\limits_{k=1}^nE[\bar{R}_k\mathfrak{D}^{(\beta)}_kZ^{(\alpha)}].\label{dlm1}
\end{align}
From Lemma \ref{fo02}, $E[Z^{(\alpha,\beta)}]=\frac{1}{2}E[F^3]=0.$ Hence, once again, we can employ the covariance formula (\ref{jfmn4}) to rewrite (\ref{dlm1}) as follows
\begin{align}
E[Ff(F)]-E[f'(F)]
&=Cov(f''(F),Z^{(\alpha,\beta)})+\sum\limits_{k=1}^nE[R_k\mathfrak{D}^{(\alpha)}_kF]
+\sum\limits_{k=1}^nE[\bar{R}_k\mathfrak{D}^{(\beta)}_kZ^{(\alpha)}]\notag\\
&=E\big[\sum\limits_{k=1}^n\mathfrak{D}_kf''(F)\mathfrak{D}^{(\gamma)}_kZ^{(\alpha,\beta)}\big]
+\sum\limits_{k=1}^nE[R_k\mathfrak{D}^{(\alpha)}_kF]
+\sum\limits_{k=1}^nE[\bar{R}_k\mathfrak{D}^{(\beta)}_kZ^{(\alpha)}]\notag
\end{align}
for any $\gamma\in [0,1].$ Furthermore, by the estimate (\ref{8uwa}), we have
\begin{align*}
|\mathfrak{D}_kf''(F)|&\leq E'_k|f''(F)-f''(T_kF)|\\
&\leq 2(2\|h'\|_\infty\vee\|h''\|_\infty) E'_k|F-T_kF|^\delta\\
&\leq 2(2\|h'\|_\infty\vee\|h''\|_\infty) (|\mathfrak{D}_kF|^{\delta}+E_k|\mathfrak{D}_kF|^{\delta}),\,\,1\leq k\leq n.
\end{align*}
Those, combined with (\ref{v34a}) and (\ref{v5b4a}), imply that
\begin{align*}
|E[Ff(F)]-E[f'(F)]|&\leq 2(2\|h'\|_\infty\vee\|h''\|_\infty) \sum\limits_{k=1}^nE|(|\mathfrak{D}_kF|^{\delta}+E_k|\mathfrak{D}_kF|^{\delta})D^{(\gamma)}_kZ^{(\alpha,\beta)}|\\
&+2^{1+\delta}(2\|h'\|_\infty\vee\|h''\|_\infty)\sum\limits_{k=1}^nE|(|\mathfrak{D}_kF|^{2+\delta}+E_k|\mathfrak{D}_kF|^{2+\delta})\mathfrak{D}^{(\alpha)}_kF|\\
&+2^{1+\delta}(2\|h'\|_\infty\vee\|h''\|_\infty)\sum\limits_{k=1}^nE|(|\mathfrak{D}_kF|^{1+\delta}+E_k|\mathfrak{D}_kF|^{1+\delta})\mathfrak{D}^{(\beta)}_kZ^{(\alpha)}|.
\end{align*}
As a consequence, by taking the supremum over all $h$ satisfying $\|h'\|_\infty,\|h''\|_\infty\leq 1,$ we deduce
\begin{align*}
d_2(F,N)&\leq  4 \sum\limits_{k=1}^nE|(|\mathfrak{D}_kF|^{\delta}+E_k|\mathfrak{D}_kF|^{\delta})D^{(\gamma)}_kZ^{(\alpha,\beta)}|\\
&+2^{2+\delta}\sum\limits_{k=1}^nE|(|\mathfrak{D}_kF|^{1+\delta}+E_k|\mathfrak{D}_kF|^{1+\delta})\mathfrak{D}^{(\beta)}_kZ^{(\alpha)}|\\
&+2^{2+\delta}\sum\limits_{k=1}^nE|(|\mathfrak{D}_kF|^{2+\delta}+E_k|\mathfrak{D}_kF|^{2+\delta})\mathfrak{D}^{(\alpha)}_kF|.
\end{align*}
So we obtain the bound (\ref{kofd4})  by replacing $F$ by $F/\sigma.$

\noindent{\it Step 3.} This step is similar to {\it Step 3} in the proof of Theorem \ref{lesser4}. We have
$$\sum\limits_{k=1}^nE|(|\mathfrak{D}_kF|^{\delta}+E_k|\mathfrak{D}_kF|^{\delta})D^{(\gamma)}_kZ^{(\alpha,\beta)}|\leq 2\big(\sum\limits_{k=1}^nE|\mathfrak{D}_kF|^{3+\delta}\big)^{\frac{\delta}{3+\delta}}
\big(\sum\limits_{k=1}^nE|D_kZ^{(\alpha,\beta)}|^{\frac{3+\delta}{3}}\big)^{\frac{3}{3+\delta}},$$
$$\sum\limits_{k=1}^nE|(|\mathfrak{D}_kF|^{1+\delta}+E_k|\mathfrak{D}_kF|^{1+\delta})\mathfrak{D}^{(\beta)}_kZ^{(\alpha)}|\leq 2\big(\sum\limits_{k=1}^nE|\mathfrak{D}_kF|^{3+\delta}\big)^{\frac{1+\delta}{3+\delta}}
\big(\sum\limits_{k=1}^nE|\mathfrak{D}_kZ^{(\alpha)}|^{\frac{3+\delta}{2}}\big)^{\frac{2}{3+\delta}},$$
$$\sum\limits_{k=1}^nE|(|\mathfrak{D}_kF|^{2+\delta}+E_k|\mathfrak{D}_kF|^{2+\delta})\mathfrak{D}^{(\alpha)}_kF|\leq 2\sum\limits_{k=1}^nE|\mathfrak{D}_kF|^{3+\delta}.$$
So the bound (\ref{kofd44}) follows from (\ref{kofd4}).

The proof of Theorem \ref{ltser4a} is complete.\hfill$\square$

\section{Appendix: Moment inequalities}\label{hj48y9}
In this Section, to make the paper self-contained, we provide some useful moment inequalities which are stated in terms of difference operators $\mathfrak{D}_k,1\leq k\leq n.$ More moment inequalities for nonlinear statistics can be found in Chapter 15 of \cite{Boucheron2013}.

\begin{prop}[Marcinkiewicz-Zygmund type inequality] Let $U=U(X)\in L^{p}(P)$ for some $p>2.$ We have
\begin{equation}\label{9hj4}
\|U\|_p^2\leq |EU|^2+ (p-1)\sum\limits_{k=1}^n\|\mathfrak{D}_kU\|_p^2,
\end{equation}
where $\|.\|_p$ denotes the norm in $L^p(P).$
\end{prop}
\begin{proof} It follows from the proof of Proposition 2.3 in \cite{dungnt2019} that we can write $U=EU+Y_1+...+Y_n,$ where $Y_k=E[\mathfrak{D}_kU|\mathcal{F}_k],\,\,1\leq k\leq n.$ Then, the inequality (\ref{9hj4}) follows directly from Theorem 2.1 in \cite{Rio2009a} and the fact that $\|Y_k\|_p\leq \|\mathfrak{D}_kU\|_p.$
\end{proof}
\begin{prop}[von Bahr-Esseen type inequality] Let $U=U(X)\in L^{1+\delta}(P)$ for some $\delta\in (0,1].$ We have
 \begin{equation}\label{k1h5}
E|U|^{1+\delta}\leq  |EU|^{1+\delta}+2^{2-\delta}\sum\limits_{k=1}^nE|\mathfrak{D}_kU|^{1+\delta}\leq  |EU|^{1+\delta}+4\sum\limits_{k=1}^nE|\mathfrak{D}_kU|^{1+\delta}.
\end{equation}
In particular, for $\delta=1,$ we have the Efron-Stein inequality that reads
\begin{equation}\label{k1h5a}
E|U|^{2}\leq |EU|^{2}+\sum\limits_{k=1}^nE|\mathfrak{D}_kU|^{2}.
\end{equation}
\end{prop}
\begin{proof} Put $h(x):=|x|^{1+\delta}$ and $V:=\int_0^1 h'(tU+(1-t)EU)dt.$ It is easy to check that
$$|h'(x)-h'(y)|=(1+\delta)||x|^\delta sign(x)-|y|^\delta sign(y)|\leq 2^{1-\delta}(1+\delta) |x-y|^\delta\,\,\forall\,\,x,y\in \mathbb{R}.$$
Hence,
$$|\mathfrak{D}_kV|=|E'_k[V-T_kV]|\leq 2^{1-\delta}E'_k|U-T_kU|^\delta \leq 2^{1-\delta}(|\mathfrak{D}_kU|^\delta+E_k|\mathfrak{D}_kU|^\delta),\,\,1\leq k\leq n.$$
By Taylor's expansion we have
$$h(U)-h(EU)=\int_0^1 h'(tU+(1-t)EU)dt(U-EU)=V(U-EU).$$
An application of Proposition \ref{lods3} gives us
\begin{align*}
E|U|^{1+\delta}-|EU|^{1+\delta}&=E[V(U-EU)]=Cov(V,U)\\
&= \sum\limits_{k=1}^nE\left[\mathfrak{D}_kV\mathfrak{D}^{(\alpha)}_kU\right]\\
&\leq 2^{1-\delta}\sum\limits_{k=1}^nE\left[(|\mathfrak{D}_kU|^\delta+E_k|\mathfrak{D}_kU|^\delta)|\mathfrak{D}^{(\alpha)}_kU|\right]\\
&\leq 2^{2-\delta}\sum\limits_{k=1}^nE|\mathfrak{D}_kU|^{1+\delta}.
\end{align*}
This finishes the proof of (\ref{k1h5}). When $\delta=1,$ we have
\begin{align*}
E|U|^2&=|EU|^2+Cov(U,U)\\
&=|EU|^2+ \sum\limits_{k=1}^nE\left[\mathfrak{D}_kU\mathfrak{D}^{(\alpha)}_kU\right]\\
&\leq |EU|^2+ \sum\limits_{k=1}^nE|\mathfrak{D}_kU|^2.
\end{align*}
So the proof of Proposition is complete. The reader can consult Section 3.1 in \cite{Boucheron2013} for the different versions of the Efron-Stein inequality.
\end{proof}

\noindent {\bf Acknowledgments.} The author thanks the anonymous referee for valuable comments for improving the paper. This research was funded by Viet Nam National Foundation for Science and Technology Development (NAFOSTED) under grant number 101.03-2019.08. A part of this paper was done while the author was visiting the Vietnam Institute for Advanced Study in Mathematics (VIASM). The author would like to thank the VIASM for financial support and hospitality.

\end{document}